\documentclass [11pt]{article}
\usepackage{amsthm}
\usepackage{amsmath}
\usepackage{amssymb}
\usepackage{enumerate}
\usepackage{epsfig}
\usepackage{srcltx}
\usepackage{rotating}

\usepackage[usenames,dvipsnames]{color}

\newcounter{contador}
\newtheorem{propo}[contador]{Proposition}
\newtheorem{teo}[contador]{Theorem}
\newtheorem{lem}[contador]{Lemma}

\newtheorem{corol}[contador]{Corollary}

\newcommand{\rec}{\noindent}    

\newcommand{\su}{{\mathbb S}^1} 

\newcommand{\K}{{\mathbb K}}
\newcommand{\R}{{\mathbb R}}
\newcommand{\C}{{\mathbb C}}
\newcommand{\N}{{\mathbb N}}

\newcommand{\V}{{\cal{V}}}
\newcommand{\U}{{\cal{U}}}

\newcommand{\G}{{\cal{G}}}

\title{Lie Symmetries of Birational Maps Preserving\\ Genus $0$ Fibrations.\footnote{{\bf Acknowledgements.}
 The second author is partially supported by
Ministry of Economy and Competitiveness of the Spanish Government
through grant DPI2011-25822 and grant 2014-SGR-859 from AGAUR,
Generalitat de Catalunya.}}

\author{Mireia Llorens$^{(1)}$ and V\'{\i}ctor Ma\~{n}osa $^{(2)}$
  \\*[.1truecm]
\\*[-.2truecm]{\small \textsl{$^{(1)}$ Dept. de Matem\`{a}tiques,}}
\\*[-.2truecm] {\small \textsl{Universitat Aut\`{o}noma de Barcelona,}}
\\*[-.2truecm] {\small \textsl{08193 Bellaterra, Barcelona, Spain}}
\\*[-.2truecm] {\small \texttt{mllorens@mat.uab.cat}}
\\*[-.2truecm] {\small \textsl{$^{(2)}$ Dept. de Matem\`{a}tica Aplicada III,}}
\\*[-.2truecm] {\small \textsl{Control, Dynamics and Applications Group (CoDALab),}}
\\*[-.2truecm] {\small \textsl{Universitat Polit\`{e}cnica de Catalunya}}
\\*[-.2truecm] {\small \textsl{Colom 1, 08222 Terrassa, Spain}}
\\*[-.2truecm] {\small \texttt{victor.manosa@upc.edu}}
}
\begin{document}

\maketitle
\begin{abstract}
We prove that any planar birational integrable map, which preserves
a fibration given by genus $0$ curves has a Lie symmetry and some
associated invariant measures. The obtained results allow to study
in a systematic way the global dynamics of these maps. Using this
approach, the dynamics of several maps is described. In particular
we are able to give, for particular examples, the explicit
expression of the rotation number function, and the set of periods
of the considered maps.


\end{abstract}

\rec {\sl 2010 Mathematics Subject Classification:} 37C25, 37C27,
 37E45,  39A23.

\rec {\sl Keywords:}  Integrable maps; Lie symmetries; periodic
orbits; rational parameterizations.

\section{Introduction}\label{S_intro}

A planar \emph{rational} map $F:\U\to\U$, where $\U\subseteq\K^2 $
is an open set and where $\K$ can either be $\R$ or $\mathbb{C}$, is
called \emph{birational} if it has a rational inverse $F^{-1}$
defined in $\U$. Such a map is \emph{integrable} if there exists a
non-constant function $V:\U\to\K$ such that $ V(F(x,y))=V(x,y),$
which  is called a \emph{first integral} of $F$. If a map $F$
possesses a first integral $V$, then the level sets of $V$ are
invariant under $F$. We say that a map $F$ preserves a fibration of
curves $\{C_h\}$ if each curve  $C_h$ is invariant under the
iterates of $F$.

In this paper, we consider integrable birational maps that have
\emph{rational first integrals}
$$
V(x,y)=\frac{V_1(x,y)}{V_2(x,y)},$$ so that they preserve the
fibration given by the algebraic curves
$$
V_1(x,y)-h\,V_2(x,y)=0,\,\mbox{for } h\in\mathrm{Im}(V).
$$

The dynamics of planar birational maps on with rational first
integrals can be classified in terms of the \emph{genus} of the
curves in the fibration associated to the first integral. In
summary, it is known that any birational map $F$ preserving a
fibration of \emph{nonsingular} curves of generic genus greater or
equal than $2$ is  globally periodic. This is a consequence of
Hurwitz automorphisms theorem and Montgomery periodic homeomorphisms
theorem. If the genus of the preserved fibration is generically $1$,
then either $F$ or $F^2$ are conjugate to a linear action. The
reader is referred to \cite{DF,GM,JRV} and references therein for
more details.

We present a systematic way to study the case where
$\{V=h\}_{h\in\mathrm{Im}(V)}$ is a \emph{genus $0$ fibration}, that
is, when each curve $\{V=h\}$ have genus $0$. As it is shown in the
next section, by using rational parameterizations of the curves of
the invariant fibration one obtains that on each curve any
birational map is conjugate to a M\"obius transformation. Using this
fact, we will prove that any of these maps possesses a Lie symmetry
with an associated invariant measure. The proofs are constructive,
so we are able to give the explicit expression of the symmetry and
the density of the measure for particular examples. These results
permits us to give a global analysis of the dynamics of the maps
under consideration. In particular, this approach allows to obtain
the explicit expression of the rotation number function associated
to the maps defined  on an open set foliated by closed invariant
curves. The explicitness of the rotation number function is a
special feature of these kind of maps, on the contrary of what
happens when the maps preserve genus $1$ fibrations,   and it
facilitates the full characterization of the set of periods of the
maps.

The paper is structured in two sections. In Section \ref{S_Main} we
present the main results of the paper: the ones that ensures the
existence of Lie symmetries and invariant measures, Theorem
\ref{T_Liesym} and Corollary \ref{C_Mesura} respectively, and the
one that characterizes the existence of conjugations of birational
maps with conjugated associated M\"obius maps, Proposition
\ref{P_conjgacions}.  In Section \ref{S_applic}, we apply our
approach to study the global dynamics of some particular maps and
recurrences.

\section{Main results}\label{S_Main}

\subsection{Dynamics via associated M\"obius
transformations}\label{Ss_Moeb}

From now on we assume that $F$ is a birational map with a rational
first integral $V=V_1/V_2$, where  both $F$ and $V$ are defined in
an open set of the domain of definition of the dynamical system, the
\emph{good set} $\mathcal{G}(F)\subset \K^2$. We also assume that
there exists a set $\mathcal{H}\subset \mathrm{Im}(V)\subset \K$
such that the set $\U:=\{(x,y)\in\K^2$: $V(x,y)\in\mathcal{H}\}\cap
\mathcal{G}(F)$ is a non empty open set of $\K^2$, and each curve
$C_h:=\{V_1(x,y)-hV_2(x,y)=0\}$ is irreducible in $\C$ and has genus
$0$.

A characteristics  of genus $0$ curves is that they are rationally
parameterizable. Recall that a \emph{rational} curve $C$ in $\C^2$,
is an algebraic one such that there exists rational functions
$P_1(t),P_2(t)\in\C(t)$ such that for almost all of  the values
$t\in
  \K$ we have $(P_1(t),P_2(t))\in C$; and reciprocally, for almost all point $(x,y)\in C$, there exists $t\in \K$
  such that $(x,y)=(P_1(t),P_2(t))$.
The map $P(t)=(P_1(t),P_2(t))$ is called a \emph{rational
parametrization} of $C$.

The relationship between rational curves and genus $0$ ones is given
by the Cayley-Riemann Theorem which states that an algebraic curve
in $\C^2$ is rational if and only if has genus $0$, see \cite[Thm.
4.63]{SWPD} for instance.

Of course a rational curve has not a  unique rational
parametrization, however one can always obtain a \emph{proper} one,
that is, a parametrization $P$ such that  $P^{-1}$ is also rational.
In other words, one can always find  a birational parametrization of
a genus $0$ curve $C$, which is unique modulus M\"obius
transformations, \cite[Lemma 4.13]{SWPD}.

Observe that given a birational map $F$ preserving an algebraic
curve $C$ defined in $\K^2$, by using homogeneous coordinates, one
can always extend them to a polynomial map $\widetilde{F}$ and an
algebraic curve $\widetilde{C}$ in $\K P^2$. Our main tool, is the
following known result:

\begin{propo}\label{P_mainpropo}
Let $F$ be a birational map defined on an open set $\U\subset\K^2$,
preserving the fibration given by algebraic curves of genus $0$,
$\{\widetilde{C}_h\}_{h\in \mathcal{H}}$, where $\mathcal{H}$ is an
open set of $\K$. Then for each $h\in\mathcal{H}$,
$\widetilde{F}_{|\widetilde{C}_h}$ is conjugate to a M\"obius
transformation in $ \widehat{\K}=\K \cup \{\infty\}$.
\end{propo}

\begin{proof}
Set $\K=\C$, since each curve $\widetilde{C}_h$ has a genus $0$, by
the Cayley-Riemmann Theorem there is a rational parametrization of
$\widetilde{C}_h$, $P_h:\widehat{\C}\to \widetilde{C}_h\subset\C
P^2$. This parametrization can be chosen to be proper, so that
$P^{-1}$ is also rational. In consequence, the map
$\widetilde{M}_h=P_h^{-1} \circ \widetilde{F}_h \circ P_h$, defined
in $\widehat{\C}$, is a one-dimensional complex birational map,
hence a M\"obius transformation.

Consider now the case $\K=\R$. Since every rational real curve can
be properly parameterized over the reals \cite[Thm. 7.6]{SWPD}, then
any real birational map preserving a fibration of real algebraic
curves of genus $0$ can be represented by real M\"obius
transformations.
\end{proof}

The dynamics of M\"obius transformations is well understood, see
Proposition \ref{P_dinamica-moeb} for instace. As a consequence of
the above result, a priori, for planar birational integrable maps
preserving a fibration given by genus $0$ curves, there can only
appear curves with periodic orbits of arbitrary large period, curves
filled with dense solutions, as well as curves with one or two
attractive and/or repulsive points, which could be located at the
infinity line of the projective space $\K P^2$, as in the case the
case studied in Proposition \ref{P_exBR-p1}.

\subsection{Existence of Lie symmetries and invariant measures}\label{Ss-Lie-i-mesura}

Some   objects that can be associated to integrable birational maps
preserving a genus $0$ fibration are \emph{Lie symmetries} and
\emph{invariant measures}. Formally, a Lie symmetry of a map $F$
defined in an open set $\U$ is a vector field $X$, also defined in
$\U,$ such that for any $p\in \U$ it is satisfied the compatibility
equation
\begin{equation}\label{E_Lie-sym-char}
X(F(p))=DF(p)\,X(p),
\end{equation}
see \cite{CGM08,HBQC}. From a dynamical viewpoint, if $F$ is an
integrable map with first integral $V$, and therefore preserving the
associated fibration $\{C_h\}$ (see Section \ref{Ss_Moeb}), a Lie
symmetry of $F$ is a vector field $X$ with the same first integral
$V$, such that the map $F$ can be seen as the flow $\varphi$ of this
vector field at certain time $\tau(h)$, which only depends on each
invariant curve $C_h$, that is: $F(p)=\varphi(\tau(h),p)$ for all
$p\in C_h\cap \mathcal{G}(F)$.

We prove that for birational maps preserving genus $0$ fibrations
these vector fields exists, and one of them can be explicitly
constructed by using a family of parameterizations of the invariant
curves.

\begin{teo}\label{T_Liesym}
Any birational map $F$ with a rational first integral $V$,
preserving the genus $0$ fibration given by  $\{V=h\}_{h\in
\mathcal{H}}$, has a Lie symmetry. Furthermore,  if
$\{P_h(t)\}_{h\in \mathcal{H}}$ is a family of proper
parameterizations of $\{V=h\}_{h\in \mathcal{H}}$, then there is a
Lie symmetry of $F$ in $\U$ given by
\begin{equation}\label{E_Liesym}
X(x,y)=\left.D P_h\left(P_h^{-1}(x,y)\right)\cdot
Y_h\left(P_h^{-1}(x,y)\right)\right|_{h=V(x,y)}
\end{equation}
where
\begin{equation}\label{E_LiesymMoeb}
    Y_h(t)=\left(-b(h)+(d(h)-a(h))t+c(h)t^2\right)\,\dfrac{\partial}{\partial t},
\end{equation}
and  the functions $a,b,c$ and $d$ are defined by the coefficients
of the map
$$
M_h(t)=P_h^{-1}\circ F\circ P_h(t)=\frac{a(h)t+b(h)}{c(h)t+d(h)}.
$$
\end{teo}

In the next section, by using the formula given in Equation
(\ref{E_Liesym}), we will obtain the explicit Lie symmetries of some
particular examples of birational maps, see for instance Proposition
\ref{P_SL-BR}, Corollary \ref{C_SL-Saito-Saitoh}, or Section
\ref{Ss_Lie_Symm_Pal}. Prior to prove Theorem \ref{T_Liesym}, we
present the following preliminary result:

\begin{lem}\label{L_LiesymMoebiu}
A M\"obius map $M:\K\to\K$ given by
$$
M(t)=\frac{at+b}{ct+d},
$$
admits the Lie symmetry given by the vector field
$Y(t)=\left(-b+(d-a)t+ct^2\right) \dfrac{\partial}{\partial t}$.
\end{lem}

The proof of the above Lemma is straightforward. We obtained the
vector field $Y$ constructively by searching a polynomial one, but
to prove the result it is enough to check that it satisfies the
compatibility equation  $Y(M(t))=M'(t)Y(t)$.

\begin{proof}[Proof of Theorem \ref{T_Liesym}] From Proposition
\ref{P_mainpropo} we have that for each curve $C_h$ the map
$F_{|C_h}$ is conjugate to the M\"obius transformation
\begin{equation}\label{E-aux-Moeb-prova-LS}
M_h(t)=P_h^{-1}\circ F_{|C_h}\circ
P_h(t)=\frac{a(h)t+b(h)}{c(h)t+d(h)}.
\end{equation}
By Lemma \ref{L_LiesymMoebiu}, the map $M_h$ has the Lie symmetry
$Y_h$ given by (\ref{E_LiesymMoeb}). Now, by taking the differential
of each parametrization $P_h(t)$ we get that for each $(x,y)\in
C_h\cap \mathcal{G}(F)$:
$$
X(x,y)=DP_h(t)Y_h(t)_{\small \left|\begin{array}{l}
                   t=P_h^{-1}(x,y) \\
                   h=V(x,y)
                 \end{array}\right.}=DP_h(P_h^{-1}(x,y))Y_h(P_h^{-1}(x,y))_{\small \left|\begin{array}{l}
                                      h=V(x,y)
                 \end{array}\right.}.
$$
Observe that, by construction, $X$ is a vector field tangent to each
curve $C_h$, so $V$ is a first integral of $X$.

Now we check that $X$ satisfies the compatibility equation
(\ref{E_Lie-sym-char}). Indeed, set $h\in\mathcal{H}$ and take
$(x,y)\in C_h$, then by Equation (\ref{E-aux-Moeb-prova-LS}) we have
$P_h^{-1}(F(x,y))=M_h(P_h^{-1}(x,y))$. On the other hand, since
$Y_h$ is a Lie symmetry of $M_h$ we have
$Y(M_h(P_h^{-1}(x,y)))=M_h'(P_h^{-1}(x,y))\,Y(P_h^{-1}(x,y))$, hence
$$\begin{array}{rl}
    X(F(x,y))&  =DP_h(P_h^{-1}(F(x,y)))\,Y(P_h^{-1}(F(x,y)))\\
   &=DP_h(M_h(P_h^{-1}(x,y)))\,Y(M_h(P_h^{-1}(x,y)))\\
   &=DP_h(M_h(P_h^{-1}(x,y)))\,M_h'(P_h^{-1}(x,y))\,Y(P_h^{-1}(x,y))).
  \end{array}
$$
Notice that, again by Equation (\ref{E-aux-Moeb-prova-LS}), we have
$$DF(x,y)
    =DP_h(M_h(P_h^{-1}(x,y)))\,M_h'(P_h^{-1}(x,y))\,DP_h^{-1}(x,y),
$$
hence
$$
DF(x,y)\,DP_h(P_h^{-1}(x,y))=DP_h(M_h(P_h^{-1}(x,y)))\,M_h'(P_h^{-1}(x,y)),
$$
and therefore
$$
X(F(x,y))=DF(x,y)\,DP_h(P_h^{-1}(x,y))\,Y(P_h^{-1}(x,y))= DF(x,y)\,
X(x,y).
$$
\end{proof}

Notice that Theorem \ref{T_Liesym} is an step towards a positive
answer of the following question: \emph{Given a birational
integrable map with rational first integrals, does always exists a
rational Lie symmetry?} \cite[Open problem, p.252]{CM}.

Lie symmetries are interesting objects in the theory of discrete
integrability, \cite{HBQC}. In particular, their existence implies
that the one-dimensional dynamics of  real maps $F$ on an invariant
curve is essentially linear, \cite{CGM08}. Hence, as a consequence
of the above result and Theorem 1 in \cite{CGM08}, we obtain the
following result:

\begin{corol}\label{C_Lineals}
Let $F$ be a real birational map with a rational first integral $V$,
preserving the associated genus $0$ fibration given by
$\{C_h\}_{h\in \mathcal{H}}$, and  let $X$ be a Lie symmetry of $F$
with first integral $V$. Let $\gamma_h$ be a connected component of
$C_h\cap\mathcal{G}(F)$, then:
\begin{enumerate}[(a)]
  \item   If  $\gamma_h$ of $C_h$ is homeomorphic to $\su$, then
 $F_{|\gamma_h}$ is conjugate to a rotation with rotation
number given by $\tau(h)/T(h),$ where $T(h)$ is the period of $C_h$
as a periodic orbit  of (\ref{E_Liesym}), and $\tau(h)$ is defined
by the equation $F(x,y)=\varphi(\tau(h),(x,y))$, where $\varphi$
denotes the flow of $X$, and $(x,y)\in C_h$.
  \item  If $\gamma_h$ is homeomorphic to $\R$, then $F_{|\gamma_h}$ is conjugate to a translation.
  \item If $\gamma_h$ is a point, then it is a fixed point of
  $F$.
\end{enumerate}
\end{corol}

Recall that a map $F$ defined on an open set of $\R^2$, preserves a
measure absolutely continuous with respect the Lebesgue's one with
non-vanishing density $\nu$,  if $m(F^{-1}(B))=m(B)$ for any
Lebesgue measurable set $B$, where $ m(B)=\int_{B} \nu(x,y)\,dxdy$.
The existence of a Lie symmetry preserving the integral $V$
guarantees the existence of certain preserved measures as a
consequence of the results in \cite{CGM08}:

\begin{corol}\label{C_Mesura}
Let $F$ be a real birational map $F$, with a rational first integral
$V$, preserving the genus $0$ fibration given by  $\{V=h\}_{h\in
\mathcal{H}}$. Then there are some open disjoint sets $\U^+$ and
$\U^-\subset \U$ (possibly some of them, but not both, empty) such
that:
\begin{enumerate}[(a)]
  \item If $F$ preserves orientation on $\U$, then on each set $\U^+$ and $\U^-$,  it preserves  an invariant
  measure absolutely continuous with respect the Lebesgue one.
  \item On each set $\U^+$ and $\U^-$, the map $F^2$ preserves  an invariant
  measure absolutely continuous with respect the Lebesgue one.
\end{enumerate}
\end{corol}
\begin{proof}
Observe that $F$ has the Lie symmetry given by Equation
(\ref{E_Liesym}). By construction this symmetry preserves the curves
$\{V=h\}$ for each $h\in\mathcal{H}$, hence it must be a multiple of
the hamiltonian vector field associated to $V$, that is
\begin{equation}\label{E-hamiltonian}
{X}(x,y)=\mu(x,y)\,\left(- V_y(x,y)\frac{\partial}{\partial
x}+V_x(x,y) \frac{\partial}{\partial y}\right).
\end{equation}
In this case, the compatibility  equation (\ref{E_Lie-sym-char}), is
equivalent to
\begin{equation}\label{E_mu}
    \mu(F(x,y))=\det(DF(x,y))\,\mu(x,y),
\end{equation}
see \cite[Theorem 12 (ii)]{CGM08}. Set $\U^{\pm}:=\{(x,y)\in\U$:
$\pm\mu(x,y)>0\}$. Observe that some of them, but not both, can be
empty.

Suppose now that $F$ preserves orientation in $U$, then Equation
(\ref{E_mu}) implies that both sets $\U^+$ and $\U^-$ are invariant
by $F$. Taking
$$\nu_{\pm}(x,y):=\pm\dfrac{1}{\mu(x,y)},$$
we get some invariant measures on each set $\U^\pm$ given by $
m_{\pm}(\mathcal{B})=\int_{\mathcal{B}} \nu_{\pm}(x,y)dxdy$ for any
Lebesgue measurable set $\mathcal{B}\subset \U^\pm$. Indeed, take
$\mathcal{B}\in \U^+$ a Lebesgue measurable set, then by using the
change of variables formula and Equation (\ref{E_mu}) we get:
$$\begin{array}{rl}
m_+(F^{-1}(\mathcal{B}))&= \displaystyle{\int}_{F^{-1}(\mathcal{B})}
\nu_{+}(x,y)dxdy=\displaystyle{\int}_{\mathcal{B}}
\nu_+(F(x,y))\,\det(DF(x,y))dxdy =\\
     & =\displaystyle{\int}_{\mathcal{B}} \dfrac{1}{\mu(F(x,y))}\,\det(DF(x,y))dxdy
=\displaystyle{\int}_{\mathcal{B}}
\dfrac{1}{\mu(x,y)}\,dxdy=m_+(\mathcal{B})
  \end{array}
\
$$
In an analogous way we can prove that $m_-$ is an invariant measure
on $\U^-$, thus proving (a), and that $
m(\mathcal{B})=\int_{\mathcal{B}} 1/\mu(x,y)dxdy$ is an invariant
measure of $F^2$ in $\U^-\cup\U^+$, thus proving (b).
\end{proof}

\subsection{Detection of conjugations via conjugations of M\"obius
maps}\label{}

Another consequence of Proposition \ref{P_mainpropo}  is that it is
possible to easily verify if two birational integrable maps
preserving a genus $0$ fibration  are conjugate, by checking if
their associated M\"obius transformations are conjugate. This is
summarized in the result below which, additionally, allows to
construct the explicit conjugations. In Section \ref{Ss_Pal_conju}
we use this result to detect the conjugations between  the maps
associated to the six recurrences presented by F.~Palladino in
\cite{P}.

Prior to state the result we introduce the main assumptions and
notation. In the following we  assume that $F$ and $G$ are
integrable birational maps in $\K$ with first integrals $V$ and $W$
respectively, and that there exists some nonempty open sets
$\mathcal{H}\subseteq \mathrm{Im}(V)$ and $\mathcal{K}\subseteq
\mathrm{Im}(W)$ such that the sets $\mathcal{U}=\{(x,y)$:
$V(x,y)=h\in\mathcal{H}\}$ and $\mathcal{V}=\{(x,y)$:
$W(x,y)=k\in\mathcal{K}\}$ are nonempty, and each curve
$C_h=\{V=h\}$ for $h\in \mathcal{H}$ and $D_k=\{W=k\}$ for $k\in
\mathcal{K}$ are irreducible in $\C$ and rational.

Let $\{P_h\}_{h\in\mathcal{H}}$ and $\{Q_k\}_{k\in\mathcal{K}}$ be
families of proper parameterizations of the family of curves
$\{C_h\}_{k\in\mathcal{K}}$ and $\{D_k\}_{k\in\mathcal{K}}$
respectively,  and let $M_h=P_h^{-1}\circ F\circ P_h$ and
$N_{k}=Q_k^{-1}\circ G\circ Q_k$ denote the M\"obius transformations
associated to $F_{|C_h}$ and $G_{|D_k}$, respectively. Let
$\mathcal{G}(M_h)$ denote the good set of $M_h$ in $\K$.

\begin{propo}\label{P_conjgacions} Under the previous
assumptions,  $F$ is conjugate with $G$  via a conjugation
$F=\Psi^{-1}\circ G\circ \Psi$, where $\Psi$ is a correspondence
between the curves $C_h$ in $\U$ and the curves $D_k$ in $\V$, if
and only if there exists a correspondence $f$ between $\mathcal{H}$
and $\mathcal{K}$, such that for all $h\in\mathcal{H}$ there exists
an invertible map $m_h$ defined in $\mathcal{G}(M_h)$ such that
$M_h=m_h^{-1}\circ N_k \circ m_h$ for $k=f(h)$. Furthermore, the
conjugation is given by:
\begin{equation}\label{E_conjugacions}
\Psi(x,y)=\left. Q_{f(h)}\circ m_{h}\circ
P_h^{-1}(x,y)\right|_{\small
                   h=V(x,y)}
\end{equation}
and
\begin{equation}\label{E_conjugacions2}
 \Psi^{-1}(u,v)=\left. P_{f^{-1}(k)}\circ m_{f^{-1}(k)}^{-1}\circ
Q_k^{-1}(u,v)\right|_{\small
                   k=W(u,v)}
\end{equation}
\end{propo}

\begin{proof} Suppose that there exists a bijection $k=f(h)$  between $\mathcal{H}$
and $\mathcal{K}$, such that for all $h\in\mathcal{H}$ there exists
an invertible map $m_h$ defined in $\mathcal{G}(M_h)$ such that
$M_h=m_h^{-1}\circ N_k \circ m_h$ for $k=f(h)$. Consider the map
$\Psi$ given by Equation (\ref{E_conjugacions}), it is easy to check
that, by construction,  it maps  any curve $C_h$ to the curve
$D_{f(h)}$. Hence $W(\Psi(x,y))=f(h)$ for each $(x,y)\in C_h$.

Set $\Phi(u,v)=\left. P_{f^{-1}(k)}\circ m_{f^{-1}(k)}^{-1}\circ
Q_k^{-1}(u,v)\right|_{\small
                   k=W(u,v)}.$
Now we prove that $\Phi\circ \Psi=\mathrm{Id}$. Indeed, take
$(x,y)\in C_h$, set $k:=f(h)=W(\Psi(x,y))$, then
$$
 \Phi\circ \Psi (x,y) =P_{f^{-1}(k)}\circ m_{f^{-1}(k)}^{-1}\circ
Q_{k}^{-1}  \circ Q_{f(h)}\circ m_{h}\circ P_{h}^{-1}(x,y)=$$$$
=P_{h}\circ m_{h}^{-1}\circ Q_k^{-1} \circ Q_{k}\circ m_{h}\circ
P_{h}^{-1}(x,y)=(x,y),
$$
and analogously $ \Psi\circ \Phi=\mathrm{Id}$, hence
$\Psi^{-1}=\Phi$.

Now we prove that $\Psi$ is a conjugation. Indeed,
$$
\Psi^{-1}\circ G\circ \Psi(x,y)= P_{h}\circ m_{h}^{-1}\circ
Q_k^{-1}\circ G\circ Q_{k}\circ m_{h}\circ P_{h}^{-1}(x,y)=
$$
$$
= P_{h}\circ m_{h}^{-1}\circ N_k \circ m_{h}\circ P_{h}^{-1}(x,y)=
 P_{h}\circ M_h\circ P_{h}^{-1}(x,y)=F(x,y).
$$

Conversely, suppose that  $F$ is conjugate with $G$ in $\U$, via a
conjugation $F=\Psi^{-1}\circ G\circ \Psi$ such that for all
$h\in\mathcal{H}$, there exists $k\in\mathcal{K}$ such that
$\Psi(C_h)=D_k$ and reciprocally, for all $k\in\mathcal{K}$, there
exists $h\in\mathcal{H}$ such that $\Psi^{-1}(D_k)=C_h$. This fact
allows to introduce a correspondence $f$ between $\mathcal{H}$ and
$\mathcal{K}$ via $k=W\circ \Psi_{|C_h}$ and $h=V\circ
\Psi^{-1}_{|D_k}$.

Take $(x,y)\in C_h$, set $k:=f(h)=W(\Psi(x,y))$, then
$$
F(x,y)=P_h\circ M_h\circ P_h^{-1}(x,y)= \Psi^{-1}\circ G \circ \Psi,
$$
and therefore for all $t=P_h^{-1}(x,y)\in\mathcal{G}(M_h)$:
$$
M_h(t)=P_h^{-1}\circ \Psi^{-1}\circ G \circ \Psi\circ P_h(t)=
P_h^{-1}\circ \Psi^{-1}\circ Q_k\circ N_k\circ Q_k^{-1} \circ
\Psi\circ P_h(t).
$$
So there exists a conjugation $m_h=Q^{-1}_{f(h)} \circ \Psi\circ
P_h(t)$ between $M_h$ and $N_k$ for $k=f(h)$.~\end{proof}

Observe that if the dependence on the parameters $h$ and $k$ in the
terms of the expression (\ref{E_conjugacions}) and
(\ref{E_conjugacions2}) is rational, and $f$ is a birational
function, then the conjugation $\Psi$ is a birational map, and
therefore global. This is the case of the conjugations constructed
in the proof of Proposition \ref{P_conju-palladino}. However, notice
that, this is not the general situation since not every proper
parametrization of a curve of the form $\{V=h\}$ has a rational
dependence on the parameter $h$. This is the case of the ones
obtained when we use the parametrization algorithms implemented in
some computer algebra software, \cite{H}.

\section{Applications}\label{S_applic}

In this section we show how the results in Section \ref{S_Main} can
be used to analyze the global dynamics of birational maps preserving
genus $0$ fibrations in a unified way. The considered examples
include a one-parameter family of maps previously studied by
G.~Bastien and M.~Rogalski in \cite{BRcon}, a map introduced by
S.~Saito and N.~Saitoh in \cite{SS}, and the maps associated to some
difference equations considered by F.~Palladino in \cite{P}. The
method can also be applied to study some  other maps appearing in
the literature, for instance the ones in \cite{BR2} and \cite{SG}.

\medskip

As we will need to recall it in every example, and although it is
well-known, for convenience of the reader we summarize the dynamics
of M\"obius maps in the following well-know result, see
\cite[Section 2.2]{CGM1} for instance.

\begin{propo}\label{P_dinamica-moeb}
Consider the map $M(t)=(at+b)/(ct+d)$, where $a,b,c,d\in\K$, with
$c\neq 0$, defined for  $t\in \widehat{\K}$. Set
$\Delta=(d-a)^2+4bc,$ and
$\xi=(a+d+\sqrt{\Delta})/(a+d-\sqrt{\Delta})$.
\begin{enumerate}
  \item[(a)] If $\Delta\neq 0$, then there are two fixed points
  $t_0$ and $t_1$ in $\widehat{\K}$, furthermore
\begin{enumerate}
  \item[$(a_1)$] When $|\xi|\neq 1$, one of the fixed points,  say $t_j$, is
  an attractor of $M$ in $\widehat{\K}\setminus\{t_{j+1\, (\mathrm{mod}\, 2)}\}$.
    \item[$(a_2)$] When $|\xi|=1$, then $M$ is conjugated to a
    rotation in $\widehat{\K}$ with rotation number
    $\theta:=\arg(\xi)\, (\mathrm{mod}\,
    2\pi)$. In particular, $M$ is periodic with minimal period $p$ if and only if $\xi$ is a primitive root of the unity.
     Furthermore, $W(t)=\left|(t-t_0)/(t-t_1)\right|$ is a first integral of $M$.
\end{enumerate}
  \item[(b)] If $\Delta=0$, then there is a unique fixed point $t_0$ which is
  a    global attractor in
  $\widehat{\K}$ of $M$.
\end{enumerate}
\end{propo}

\subsection{The Bastien and Rogalski map
revisited}\label{S_BR}

We consider the planar birational map defined in $\R^{2,+}$, given
by
\begin{equation}\label{E_BR}
F_a(x,y)=\Big(y,\frac{a-y+y^2}{x}\Big)
\end{equation}
with $a>1/4$. The periodic structure  of this map was characterized
by G.~Bastien and M.~Rogalski as a part of the study of the
difference equation $u_{n+2}=(a-u_{n+1}+u_{n+1}^2)/u_n$, see
\cite{BRcon}. The map $F_a$ possess the rational first integral
$$V_a(x,y)=\frac{x^2+y^2-x-y+a}{x y},$$ and
it preserves the fibration of $\R^{2,+}$ given by the algebraic
curves of genus 0 (conics):
\begin{equation}\label{E_foliacioBR}
C_h=\{x^2+y^2-x-y+a-h x y=0\},
\end{equation}
where $h\in\left(2-{1}/{a},\infty\right)$. The map has a unique
fixed point in $\R^{2,+}$, given by $(a,a)$ with energy level
$h_c:=V(a,a)=2-{1}/{a}$. This point is an elliptic one since
$a>1/4$. Let $\widetilde{C}_h=\{[x:y:z],\,x^2+y^2-h x
y-xz-yz+az^2=0\}$ and
$\widetilde{F}_a([x:y:z])=[xy:az^2-yz+y^2:xz]$, denote the
extensions of $C_h$ and $F_a$  to $\R P^2$. We prove:
\begin{propo}\label{P_exBR-p1}
Set $a>1/4$ and $h_c=2-1/a$, then the following statement hold:
\begin{enumerate}[(a)]
  \item For $h>2$, $C_h$ is a hyperbola,
          $F_{a|C_h}$ is conjugated to a translation, and there are
          two fixed points of $\widetilde{F}_{a|\widetilde{C}_h}$ at
          infinity given by $[2:h\pm\sqrt{h^2-4}:0]$,  an attractor and a repeller.
  \item For $h=2$, $C_h$ is a parabola,
          $F_{a|C_h}$ is conjugated to a translation, and the point at
          the infinity
          $[1:1:0]$ is a global attractor of $\widetilde{F}_{a|\widetilde{C}_h}$.
  \item For $h_c<h<2$, $C_h$ is an ellipse,
          $F_{a|C_h}$ is conjugated to a rotation with  rotation number
         \begin{equation}\label{E_rotnumBR}
\theta(h)=\arg\Bigg(\frac{h-i\sqrt{4-h^2}}{2}\Bigg)\, (\mathrm{mod}
\, 2\pi).
\end{equation}
\end{enumerate}
\end{propo}

Using the rotation number function (\ref{E_rotnumBR}) we reobtain
the results of Bastien and Rogalski in \cite{BRcon}, concerning the
set of periods of each particular map $F_a$, and of the family of
maps.

\begin{propo}[\cite{BRcon}]\label{P_exBR-p2}
Set $a>1/4$, $h_c=2-1/a$, and
$\theta_a=\frac{1}{2\pi}\arg\left(\frac{2a-1-i\sqrt{4a-1}}{2a}\right).$
Then:
\begin{enumerate}[(a)]
\item For any fixed $a>1/4\,$ and any natural number $p\geq \mathrm{E}\left(1/(1
-\theta_a)\right)+1$ there exists $h_p\in(h_c,2)$ such that
$C_{h_p}$ is filled of $p$-periodic orbits.
\item For all $p\in\N$, $p\geqslant3$ there exists
$a>1/4$ and $h_p\in(h_c,2)$ such that $C_{h_p}$ is filled of
$p$-periodic orbits.
\end{enumerate}
        \end{propo}

Prior to prove the above results we find a real proper
parametrization the curves (\ref{E_foliacioBR}). In our case, we
will use the method of \emph{parametrization by lines} \cite[Section
4.6]{SWPD}, to obtain the proper parametrization. First, for each
curve in (\ref{E_foliacioBR}), we consider the point
$$(x_0,y_0)=\Bigg(a,\frac{a h+1+\delta}{2}\Bigg)\in C_h,$$
where $\delta=\sqrt{(ah+1)^2-4a^2}\in\R$. Taking the new variables
$x=u+x_0,y=v+y_0$ we bring this point to the origin, so that, each
curve in the new variables is defined by $f_1(u,v)+f_2(u,v)=0$ where
$f_k$ stands for the homogeneous part of degree $k$. In our case:
\begin{align*}
f_1(u,v)&=\left(- a{h}^{2}/2- \left( 1+\delta \right) h/2-1+2
a\right)u+\delta v,\\
f_2(u,v)&=u^2+v^2-h  u v.
\end{align*}
We compute the intersection points of these curves  with the the
lines $v=t\,u$,
\begin{displaymath}
 \left\{ \begin{array}{l}
 v=t u,\\
 f_2(u,v)+f_1(u,v)=0,
  \end{array} \right.
\end{displaymath}
obtaining an affine parametrization $(u(t),v(t))$, so that the
parametrization of the corresponding curve $C_h$ is the \emph{real}
one given by $P_h(t)=(P_{1,h}(t),P_{2,h}(t))=(u(t)+x_0,v(t)+y_0)$
where
         $${P_{1,h}(t)=\frac{2\delta t-ah^2-(1+\delta)h-2+4a}{2(-t^2+ht-1)}+a,}$$
         $${P_{2,h}(t)=\frac{\left( -ah+\delta-1 \right) {t}^{2}+ \left( 4\,a-2 \right) t-ah-
\delta-1}{2(-t^2+ht-1)}.}$$

It is straightforward to check (see Theorem \ref{T_proper-or-not} in
the Appendix) that for each $h>h_c$, the above parameterizations of
$C_h$ is a proper one. Furthermore, by using the method described in
Theorem \ref{T_inversa-param} and Equation
(\ref{E_inversa-general}), one gets that its inverse is given by
         $${{P_h}^{-1}(x,y)=
{\frac {-2\,\delta\,x+ \left( a{h}^{2}+(\delta+1)h-4\,a+2 \right)
y-ah+2 \,a+\delta-1}{ \left( a{h}^{2}+(1-\delta)h-4\,a+2 \right)
x+2\,\delta \,y+ah-2\,a-\delta+1}}.}$$

The parametrization $P_h(t)$ is defined in $\widehat{\R}$ if
$h_c<h<2$, and in $\widehat{\R}\setminus\{t_0,t_1\}$ where
\begin{equation}\label{E_puntstiBR}
t_{j}:=\frac{h+(-1)^j\sqrt{h^2-4}}{2},
\end{equation}
if $h\geq 2$.
 Observe
that this values of the parameters do not correspond to affine
points of $C_h$. Indeed, each curve extends to $\R P^2$ as
$\widetilde{C}_h=\{[x:y:t],\,x^2+y^2-h x y-xt-yt+at^2=0\}$, which is
properly parameterized for $t\in\widehat{\R}$ by
$$\begin{array}{rl}
\widetilde{P}_h(t)=&\left[-2 a{t}^{2}+ \left( 2 ah+2 \delta \right)
t-a{h}^{2}-h\delta+2 a-h -2\right. :\\
&\left.\left( -ah+\delta-1 \right) {t}^{2}+ \left( 4 a-2 \right)
t-ah- \delta-1:2(-t^2+ht-1)\right].
  \end{array}
$$
The values of the parameters $t_j$ correspond with the infinite
points $$Q_{j,h}=[2: h+(-1)^j\sqrt{h^2-4}: 0]$$ which are fixed
points of the extension of $F$  to $\R P^2$ given by
$\widetilde{F}_a([x:y:z])=[xy:az^2-yz+y^2:xz]$. Of course,  if
$h=2$,  then $t_1=t_0=1$ and $Q_{1,h}=Q_{0,h}=[1:1:0]$.

As an immediate consequence of the above arguments, Theorem
\ref{T_Liesym} and Corollary \ref{C_Mesura}, we have the following
result:

\begin{propo}\label{P_SL-BR}
Each map (\ref{E_BR}) possesses the Lie symmetry
\begin{equation}\label{E-LS-BR}
X(x,y)=\frac{x^2-y^2+a-x}{y}
          \frac{\partial}{\partial x}+\frac{x^2-y^2-a+y}{x}\frac{\partial}{\partial
          y},\end{equation} and preserves a measure absolutely
continuous with respect the Lebesgue one in $\R^{2,+}$ given by
$$m(\mathcal{B})=\int_\mathcal{B} \dfrac{1}{x\,y}\,dx\,dy$$ for any
Lebesgue measurable set $\mathcal{B}$ of $\R^{2,+}$.
\end{propo}
\begin{proof}
From the above arguments,  we have that
$\{P_h(t)\}_{h\in\mathcal{H}}$, where $\mathcal{H}:=\{h>h_c\}$, is a
family of proper affine parameterizations of
$\{C_h\}_{h\in\mathcal{H}}$. Now a computation shows that
\begin{equation}\label{E_Moeb-BR}
M_h(t)=P_h^{-1}\circ F_a \circ P_h(t)=\frac{(h+1)t-1}{t+1},
\end{equation}
so $F_{|C_h}$ is conjugate to  $M_{h|\widehat{\R}}$  if $h_c<h<2$,
and $M_{h|\widehat{\R}\setminus\{t_0,t_1\}}$ otherwise. By Lemma
\ref{L_LiesymMoebiu}, each map $M_h$ has the Lie symmetry
$Y_h(t)=\left(1-ht+{t}^{2}\right) \dfrac{\partial}{\partial t}$.
Observe that, by setting $P_h(t)=(P_{1,h}(t),P_{2,h}(t))$, then
Equation (\ref{E_Liesym}) in Theorem \ref{T_Liesym} gives the vector
field $X=X_1\partial/\partial x +X_2\partial/\partial y$ where :
$$\begin{array}{l}
X_1(x,y)=\left.P_{1,h}'(P_h^{-1}(x,y))Y_h(P_h^{-1}(x,y))\right|_{h=V(x,y)}=(x^2-y^2+a-x)/y \\
X_2(x,y)=\left. P_{2,h}'(P_h^{-1}(x,y))Y_h(P_h^{-1}(x,y))
\right|_{h=V(x,y)}=(x^2-y^2-a+y)/x
  \end{array}
$$
Hence,  we get that $F_a$ has the Lie symmetry (\ref{E-LS-BR}), also
defined in $\R^{2,+}$.

Observe that using the notation in Section \ref{Ss-Lie-i-mesura},
this vector is exactly the multiple of the hamiltonian vector field
(\ref{E-hamiltonian}) associated to $V_a$, with $\mu(x,y)=xy$. So by
Corollary \ref{C_Mesura} it defines a measure absolutely continuous
with respect the Lebesgue one in $\R^{2,+}$ given by
$m(\mathcal{B})=\int_\mathcal{B} (1/\mu)\,dx\,dy$, for any Lebesgue
measurable set $\mathcal{B}$ of $\R^{2,+}$.
\end{proof}

\begin{proof}[Proof of Proposition \ref{P_exBR-p1}]
It is straightforward to check that for any fixed $a>1/4$, the
curves $C_h$ are ellipses for $h_c<h<2$, a parabola when $h=2$, and
a branch of a hyperbola when $2<h<\infty$, see \cite{BRcon}. By
using Corollary \ref{C_Lineals} and the existence of the Lie
symmetry (\ref{E-LS-BR}), we get that the affine dynamics of each
map $F_{a|C_h}$ is conjugate to a translation when $h\geq 2$, and
conjugate to a rotation when $h_c<h<2$.

As shown above, each map $\widetilde{F}_{a|C_h}$ is conjugate to the
M\"obius one $M_h$ given in (\ref{E_Moeb-BR}). Observe that for each
$h\geq 2$, the  map $M_h(t)$ has the fixed points $t_j$ given in
(\ref{E_puntstiBR}) which correspond with the infinite points
$Q_{j,h}$, which are fixed points of $\widetilde{F}_a$.

Following the notation of Proposition \ref{P_dinamica-moeb}, for the
map $M_h$ one gets $\Delta=h^2-4$ and
$\xi=(h+2-\sqrt{h^2-4})/(h+2+\sqrt{h^2-4})$. Now,
\begin{itemize}
 \item If $h>2$, then $\Delta>0$ and $\xi<1$. In this case $t_0$ is an attractor of $M_h$
 in $\widehat{\R}\setminus t_1$ and $t_1$ is a repeller, so $Q_{0,h}$ is an attractor of $\widetilde{F}_{a|\widetilde{C}_h}$
 in $\widetilde{C}_h\setminus Q_{1,h}$, and $Q_{1,h}$ a repeller.
 \item If $h=2$, then  $\Delta=0$, so the the point $t_0=h/2=1$ is
 a global attractor of $M_2$ in $\widehat{\R}$, thus $Q_{0,2}=[1:1:0]$ is a global
 attractor of $\widetilde{F}_{a|\widetilde{C}_2}$
 in $\widetilde{C}_2$.
\item If $h_c<h<2$, then $\Delta<0$, hence $M_h$ is conjugated
to a rotation in $\widehat{\R}$, with rotation number
$$
\theta(h)=\arg\left(\frac{h+2-i\sqrt{4-h^2}}{h+2+i\sqrt{4-h^2}}\right)\,(\mathrm{mod}\,
2\pi)=\arg\Bigg(\frac{h-i\sqrt{4-h^2}}{2}\Bigg)\,(\mathrm{mod}\,2\pi),
$$ and therefore $F_{a|C_h}$ is a conjugate to a rotation with the same rotation number
(recall that, in this case,  the curves $C_h$ are affine ellipses
with no points at the infinity line).
\end{itemize}~\end{proof}

\begin{proof}[Proof of Proposition \ref{P_exBR-p2}]
(a) In the proof of Proposition \ref{P_exBR-p1} we have seen that if
$h_c<h<2$, then each map $F_{a|C_h}$ is conjugate to a rotation with
rotation number given by Equation (\ref{E_rotnumBR}). It is
straightforward to check that for a fixed $a>1/4$ this function
$\theta(h)$ grows monotonically from $\theta_a$ to $1$ for
$h\in(h_c,2)$, where
$$\theta_a= \lim\limits_{h\to h_c^+}\theta(h)=\frac{1}{2\pi}
\arg\left(\frac{h_c-i\sqrt{4-h_c^2}}{2}\right)=\frac{1}{2\pi}
\arg\left(\frac{2a-1-i\sqrt{4a-1}}{2a}\right).$$  Therefore, for all
$\theta\in (\theta_a,1)$ there exists $h\in(h_c,2)$ such that
$\theta(h)=\theta$, so to characterize the set of periods of any map
$F_a$ in $\R^{2,+}$, we need to know which are the irreducible
fractions $q/p\in (\theta_a,1)$.

Observe that if an irreducible fraction $q/p\in (\theta_a,1)$, then
$1\leq q\leq p-1$ and therefore  $(p-1)/p\in (\theta_a,1)$, because
$$
\theta_a<\frac{q}{p}\leq \frac{p-1}{p}<1.
$$
Hence we only need to characterize which are integer numbers $p$
such that $\theta_a<(p-1)/p$. For such a number, one easily obtains
$p>1/(1-\theta_a)$, hence
$$
p\geq E\left(\frac{1}{1-\theta_a}\right)+1.
$$

(b) A computation shows that when  $1/4<a\leq 1/2$, $h_c$ varies
monotonically from $-2$, to $0$, and therefore $2\pi\theta_a$ is an
angle in the third quadrant that varies from $\pi$ to $3\pi/4$. If
$a>1/2$ then $h_c>0$ and $2\pi\theta_a$ is an angle of the fourth
quadrant. Hence
$$
I:=\bigcup\limits_{a>1/4} (\theta_a,1)=\left(\frac{1}{2},1\right).
$$
Observe that $1/p\in I$ for all $p\geq 3$,  and therefore the result
follows.~\end{proof}

\subsection{The Saito and Saitoh map}\label{S_SS}

In \cite{SS}, S.~Saito and N.~Saitoh considered the  map
\begin{equation}\label{Fsaito}
F(x,y)=\left(x y,\frac{y(1+x)}{1+x y}\right),
\end{equation}
defined in the open set $\mathcal{G}(F)=\C^2
\setminus\{\bigcup_{n\geq 0} F^{-n}(x,{-1}/{x})\}$, and they found a
continuum of $3$-periodic orbits. Here we describe the global
dynamics of $F$, and we notice that there are continua of periodic
orbits of minimal period $p$ for all $p\geq 2$, which have the
simple expression $y(1+x)=h_p$, being $h_p$ any primitive $p$-root
of the unity.

As already noticed in \cite{SS}, this map has the first integral
$V(x,y)=y(1+x)$, hence the affine sets $\gamma_h:=\{y(1+x)=h,\,
h\in\C\}\cap \mathcal{G}(F)$ are invariant. The extension to $\C
P^2$ of all these sets contains the affine points of $\gamma_h$ plus
the infinity points $[1:0:0]$ and $[0:1:0]$. But these infinity
points belong to the so called \emph{indeterminacy locus}
 of  the extension of $F$ to $\C P^2$,  $\widetilde{F}([x:y:z])=[xy(xy+z^2): yz^2(x+z): (xy+z^2)z^2]$.  That is:
$\widetilde{F}([1:0:0])=\widetilde{F}([0:1:0])=[0:0:0]$, see
\cite{Diller,DF} for further details. This is the reason why we will
describe the dynamics of $F$ only in the affine space $\C^2$. The
global dynamics is given by next result.

\begin{propo}\label{proposaitocomplexa} Consider the map (\ref{Fsaito}),
then:
\begin{enumerate}
   \item[(a)] Any solution with initial condition on
   $\gamma_0$ reaches the point $(0,0)$, which is a global attractor of $F_{|\gamma_0}$  in finite time.
   \item[(b)] For $h\neq 1$, on any invariant set
    $\gamma_h$ there are two fixed points of $F$, given by $Q_0=(0,h)$ and $Q_1=(h-1,1)$. Furthermore,
               \begin{enumerate}
               \item[(i)] If $|h|<1$, then the point $Q_0$ is the attractor of $F_{|\gamma_h}$ in
               $\gamma_h\setminus
               Q_1$, and $Q_{1}$ is a repeller. Conversely if
               $|h|>1$, then
               the point $Q_0$ is the repeller and $Q_{1}$ an attractor of $F_{|\gamma_h}$ in
               $\gamma_h\setminus
               Q_0$.
               \item[(ii)] If $|h|=1$, then either any initial condition $(x_0,y_0)$ in $\gamma_h$
               gives rise to a periodic orbit
               with minimal period $p\geq 2$, when $h$ is a primitive $p$-root of the unity,
               or it gives rise to an orbit which densely fills the
               set $\gamma_h \cap \{W(x,y)=W(x_0,y_0)\}$, where
               $W(x,y)=\left|x/(x-h+1)\right|$, otherwise.
               \end{enumerate}

    \item[(c)] For $h=1$,  the point $p=(0,h)$ is the unique fixed point of $F$ in $\gamma_1$,
    which is a global attractor of $F_{|\gamma_1}$.
  \end{enumerate}
\end{propo}
\begin{proof} The set $\gamma_0$ is given by the lines $x=-1$ and $y=0$.
  Observe that $(0,0)\in \gamma_0$ is a fixed point of $F$. A simple
  computation shows that $F^2(-1,y)=(0,0)$ and $F(x,0)=(0,0)$ for
  all $x,y\in \C$, which proves statement (a).

To prove statements (b) and (c) we fix  $h\neq 0$. Observe that the
set $\gamma_h$ is given
  by the intersection of the curve $C_h=\{y(1+x)=h\}$  with
  $\mathcal{G}$. Any curve $C_h$  admits
  the trivial proper affine parametrization
  $$P_h(t)=\left(t,\frac{h}{t+1}\right)\mbox{ for } t\neq -1.$$ Observe that
  parameter $t=-1$ correspond with the excluded infinity point $[0:1:0]$. Observe that $P_h^{-1}(x,y)=x$. A
  computation gives.
 $$M_h(t)=P_h^{-1}\circ F \circ P_h(t)=\frac{ht}{t+1}.$$

It is easy to check that $\gamma_h=C_h \setminus \{\bigcup_{n\geq 0}
F^{-n}(-1/(h+1),h+1)\}$ for $h\neq -1$; that $\gamma_{-1}=C_{-1}$;
and that  for any $h\in\C$, the map $F_{|\gamma_h}$ is conjugate to
the  $M_{h|\mathcal{G}(M_h)}$, where $\mathcal{G}(M_h)=\C\setminus
\{\cup_{n\geq 0} M_h^{-n}(-1)\}$  is the good set of $M_h$ in $\C$.
In particular $\mathcal{G}(M_{-1})=\C \setminus \{-1\}$.

Each map $M_h$ has two fixed points $t_0=0$ and $t_1=h-1$
corresponding to the parameters of the points $Q_0,Q_1\in \C^2$
respectively. Now, the result follows directly from Proposition
\ref{P_dinamica-moeb}. In particular,  for $|h|=1$ with $h\neq 1$,
and since $\Delta=(1-h)^2$ and $\xi=1/h$, one easily gets that the
map $M_h$ is conjugate to a rotation in $\C$ with rotation number
$$
\theta(h)=\arg\left(\frac{1}{h}\right)\,(\mbox{mod }
2\pi)=\arg\left(\bar{h}\right)\,(\mbox{mod } 2\pi).
$$
Hence the map $F_{|\gamma_h}$ has continua of periodic orbits of all
minimal periods $p\geq 2$, located at the hyperbolae  $y(1+x)=h_p$,
being $h_p$  a primitive $p$-root of the unity. Furthermore,  in
this case
$$W(x,y)=\left| \dfrac{x}{x-h+1}\right|,$$ is a first integral of
$F_{|\gamma_h}$. Otherwise, if $h$ is not a $p$-root of the unity,
then for each initial condition on $\gamma_h$, the associated orbit
of $F_{|\gamma_h}$ fills densely the set $\gamma_h \cap
\{W=W(x_0,y_0)\}$.
\end{proof}

Notice that, as a consequence of the above result, the only periodic
orbits that can be observed in $\R^2$ are the $2$-periodic orbits
that are located in the real hyperbola $y(1+x)=-1$.

Finally, as a direct consequence of Theorem \ref{T_Liesym}, we have:

\begin{corol}\label{C_SL-Saito-Saitoh}
The map (\ref{Fsaito}) possesses the Lie symmetry
$$
X(x,y)=-x\left(1+x\right)\left( y-1 \right)\frac{\partial}{\partial
x}+ xy\left( y-1 \right)  \frac{\partial}{\partial y}.$$
\end{corol}
\begin{proof}
By Lemma \ref{L_LiesymMoebiu}, the M\"obius map $M_h(t)=ht/(t+1)$
has the Lie symmetry $Y_h(t)=\left((1-h)t+t^2\right)
\dfrac{\partial}{\partial t}$. Now the Lie symmetry of $F$ is
obtained from Equation (\ref{E_Liesym}), by using the
parametrization $P_h(t)=\left(t,h/(t+1)\right)$ and taking
$h=y(1+x)$.
\end{proof}

Notice that the above vector field is a multiple of the hamiltonian
one (\ref{E-hamiltonian}) with $\mu(x,y)=x(y-1)$.

\subsection{A difference equation with the Saito and Saitoh invariant}\label{Ss_lanostra}

We consider the \emph{real} difference equation
\begin{equation}\label{ODEnostra}
u_{n+2}=\frac{u_{n+1}(1+u_n)}{1+u_{n+1}}.
\end{equation}  This
equation possesses the invariant given by
$I(u_n,u_{n+1})=u_{n+1}(1+u_n)$. We found this difference equation
when looking for a recurrence with the same invariant as the Saito
and Saitoh map. Latter, we knew that this equation is a particular
case of the six equations introduced in \cite{P}, and also
considered in Section \ref{Ss_Pal}.

For each initial condition $u_0$, $u_1$ in $\G\subset \R$, the good
set of the equation (\ref{ODEnostra}), set
$I:=I(u_0,u_1)=u_1(1+u_0)$ and
$$t_j=t_j(u_0,u_1):=\dfrac{-1+(-1)^j\sqrt{1+4\,I}}{2},
\mbox{ and } \xi=-\frac{1+2\,I+i\sqrt{-1-4\,I}}{2\,I}.$$
\begin{propo}\label{P_odenostra} Let $\{u_n\}$ be a solution of Equation (\ref{ODEnostra}) with initial
condition $u_0,u_1$ in $\G$, then:
\begin{enumerate}
  \item[(a)] If $u_1=u_0$ and  $u_0\neq -1$, then the solution is
  constant.
  \item[(b)] If   $u_1(1+u_0)\geq -1/4$ with $u_0\neq t_1$ and $u_1\neq
  t_1$, then the solution   converges to $u=t_0$.
  \item[(c)]  If    $u_1(1+u_0)< -1/4$ then the solution is either $p$-periodic,  if
  $\xi$ is a $p$-root of unity, or such that the
set of accumulation points of $\{u_n\}_{n\in\N}$ is $\R$, otherwise.
Furthermore, the Equation (\ref{ODEnostra}) has $p$-periodic
solutions for all period $p\geq 3$.
\end{enumerate}
\end{propo}

The proof of the above proposition follows directly from the study
of the planar real dynamical system given by the map
\begin{equation}\label{Fnostra}
F(x,y)=\left(y,\frac{y(1+x)}{1+y}\right),
\end{equation}
 by noticing that
$(u_n,u_{n+1})=F^n(u_0,u_1)$. This map is defined in
$\mathcal{G}(F)=\R^2\setminus\{\bigcup_{n\geq 0} F^{-n}(x,-1)\}$. As
the case of the map (\ref{Fsaito}), the map $F$ has the first
integral $V(x,y)=y(1+x)$, hence the sets $\gamma_h:=\{y(1+x)=h,\,
h\in\R\}\cap \mathcal{G}(F)$ are  invariant. Observe that the fixed
points of $F$ have the form $(x,x)$, with $x\neq -1$. For this map
we have the following result:

\begin{propo}\label{propolanostra} Consider the map (\ref{Fnostra}),
then
\begin{enumerate}[(a)]
   \item Any solution with initial condition on
   $\gamma_0$   reaches the point $(0,0)$ in finite time.
  \item[(b)] For $h\neq 0$, the dynamics in the invariant set $\gamma_h$ is
  the following:
  \begin{enumerate}
    \item[(i)] If $h>-1/4$, then there exists two fixed points of $F$ in $\gamma_h$,
    $Q_{j,h}=(x_j,x_j)$, with $j=0,1$,
    where $x_{j}=(-1+(-1)^j\sqrt{1+4h})/2$. Furthermore $Q_0$ is a
    global attractor of $F_{|\gamma_h\setminus Q_1}$, and
    $Q_1$ is a repeller.
    \item[(ii)] If $h=-1/4$, then there exists a unique fixed points of $F$ in
    $\gamma_{-\frac{1}{4}}$, $Q=(-1/2,-1/2)$
    which is a global attractor of $F_{|\gamma_{-\frac{1}{4}}}$.
    \item[(iii)] If $h<-1/4$, then either any initial condition in
    $\gamma_h$ gives rise to a periodic orbit with minimal period $p\geq
    3$ or it gives rise to an orbits which densely fills this set,
    depending on whether $(-1-2h-i\sqrt{-1-4h})/(2h)$ is a primitive $p$-root
    of the unity or not.
  \end{enumerate}
\end{enumerate}
\end{propo}

\begin{proof}
First observe that after one or two iterations any orbit with
initial
  condition in  $\gamma_0=\left\{\{x=-1\}\right.$$\left.\cup\{y=0\}\right\}\cap\mathcal{G}(F)$ reaches the point $(0,0)$.
  Consider  $h\neq 0$, and  recall that any hyperbola
  $C_h=\{y(1+x)=h\}$
  admits   the   proper parametrization
  $P_h(t)=(t,h/(t+1))$ for $t\neq -1$.
 Some computations give  that $\gamma_h=C_h \setminus \{\bigcup_{n\geq 0}
F^{-n}(-h-1,-1)\}$; that
 $$
M_h(t)=P_h^{-1}\circ F \circ P_h(t)=\frac{h}{t+1}, $$  and that
$F_{|\gamma_h}$ is conjugate to $M_{h|\mathcal{G}(M_h)}$, where
$\mathcal{G}(M_h):=\R\setminus \{\cup_{n\geq 0} M_h^{-n}(-1)\}$.

 Using Proposition \ref{P_dinamica-moeb}, and setting $\Delta=1+4h$
and $\xi=(-1-2h-\sqrt{1+4h})/(2h)$,  we have that if
 $h>-1/4$, the map $M_h$ has two fixed
points $t_{0,h}$ (an attractor) and $t_{1,h}$ (a repeller) given by
$$
t_{j,h}=\frac{-1+ (-1)^j\sqrt{1+4h}}{2}.
$$
So on each set $\gamma_h$ there are two fixed points
$Q_{0,h}=(t_{0,h},t_{0,h})$ and $Q_{1,h}=(t_{1,h},t_{1,h})\in \R^2$,
which are an attractor and a repeller of $F_{|\gamma_h}$,
respectively. If $h=-1/4$, then $M_{-\frac{1}{4}}$ has a unique
fixed point $t=-1/2$ which is a global attractor in
$\mathcal{G}(M_{-\frac{1}{4}})$, hence the map
$F_{|\gamma_{-\frac{1}{4}}}$ has a global attractor in the point
$(-1/2,-1/2)$. If $h<-1/4$, then the map $M_{h}$ defined on
$\widehat{\R}$ is conjugate to a rotation with rotation number
$$
\theta(h)=\arg\Bigg(-\frac{1+2h}{2h}-i\frac{\sqrt{-1-4h}}{2h}\Bigg)\,(\mathrm{mod}\,2\pi).
$$
An straightforward computation shows that
$$
I=\left\{\mathrm{Image}(\theta(h)),\,
h<-\frac{1}{4}\right\}=\left(\frac{1}{2},1\right).
$$
So there are irreducible fractions with denominator $p\in I$ if and
only if $p\geq 3$ and therefore the map $F$ possesses periodic
orbits for all period $p\geq 3$, and they are located in the region
$\{y(1+x)<-1/4\}$. Furthermore the sets $\gamma_{h_p}$ where are
located the periodic orbits with minimal period $p$, are those such
that $\xi_{|h=h_p}=(-1-2h_p-i\sqrt{-1-4h_p})/(2h_p)$ is a primitive
$p$-root of the unity. Observe that $F$ has not $2$-periodic orbits.
If $\xi$ is not a root of unity then any orbit in $\gamma_h$ densely
fills this set.
\end{proof}

Observe that the proof of Proposition \ref{P_odenostra} also follows
straightforwardly from the  analysis of the M\"oebius maps $M_h$ by
noticing that due to the particular form of the family of
parameterizations $P_h$, we have that $u_n=M_h^n(u_0)$.

Also notice that the Lie symmetry of the map (\ref{Fnostra}) is
given in Section \ref{Ss_Lie_Symm_Pal}.

\subsection{The Palladino's recurrences}\label{Ss_Pal}

In \cite{P}, F.~Palladino presented the analysis of the
\emph{forbidden set} of a list of six difference equations in $\C$
with rational invariants. These invariants allow the author to make
an order reduction, so that the dynamics of the equation can be
described via a family of Riccati equations. In fact all the
Palladino's equations have an associated integrable birational map
preserving a genus $0$ fibration. Hence, the order reduction
observed by Palladino in these equations are a consequence of the
general fact observed in Proposition \ref{P_mainpropo}.

In the next subsections we will compute the associated Lie
symmetries of the maps associated to the Palladino's recurrences, we
study the conjugations between these maps and, as a last example, we
study the dynamics of a representative of each set of non-conjugate
recurrences.

\subsubsection{Lie symmetries for the maps associated to the Palladino's
recurrences.}\label{Ss_Lie_Symm_Pal}

In this section we consider the maps $F_j$, $j=1,\ldots,6$,
associated to each Palladino recurrence. For these maps we give
families of proper parametrizations associated to each fibration
$C_{j,h}=\{V_{j,1}-hV_{j,2}=0\}$ where $V_j=V_{j,1}/V_{j,2}$  is the
first integral corresponding with the invariants given in \cite{P},
we give the M\"obius transformations associated to the map, the
parametrization, and the associated Lie symmetries as well.

\medskip

\noindent 1. The equation
$u_{n+2}=\frac{u_{n+1}}{1+b\,(u_n-u_{n+1})}$ with
$b\in\C\setminus\{0\}$, has the following associated objects, which
characterize its dynamics:

{\small
\begin{center}
\noindent\begin{tabular}{| l | l | l |}
  \hline
  Associated map:    & First integral:  & Parametrization of  $C_{1,h}$: \\
 $F_1(x,y)=\left(y,\dfrac{y}{1+b\,(x-y)}\right)$   &  $V_1(x,y)=\dfrac{1+bx+by+b^2xy}{y}$
 &   $P_{1,h}(t)=\left(t,\dfrac{-bt-1}{b^2t+b-h}\right)$\\
   \hline
      M\"obius map: & Lie symmetry  of $M_{1,h}$:  & Lie symmetry of $F_1$:\\
 $M_{1,h}(t)=\dfrac{-bt-1}{b^2t+b-h}$  & $Y_{1,h}=\left(1+ \left( 2\,b-h \right) t+{b}^{2}{t}^{2}\right)\,\dfrac{\partial }{\partial t}$  &
 $\begin{array}{l}
 X_1=-\dfrac{(x-y)(bx+1)}{y}\\
X_2=-b(x-y)(by+1)
\end{array}$\\
  \hline
\end{tabular}
\end{center}
}

\bigskip

\noindent 2. Objects associated with the recurrence
$u_{n+2}=\frac{u_{n}}{1+b\,(u_{n+1}-u_{n})}$ with
$b\in\C\setminus\{0\}$:

{\small
\begin{center}
\noindent\begin{tabular}{| l | l | l |}
  \hline
  Associated map:    & First integral:  & Parametrization of  $C_{2,h}$: \\
 $F_2(x,y)=\left(y,\dfrac{x}{1+b\,(y-x)}\right)$   &  $V_2(x,y)=\dfrac{1+by+xy}{xy}$
 &   $P_{2,h}(t)=\left(t,\dfrac{1}{(h-1)t-b}\right)$\\
   \hline
      M\"obius map: & Lie symmetry  of $M_{2,h}$:  & Lie symmetry of $F_2$:\\
 $M_{2,h}(t)=\dfrac{1}{(h-1)t-b}$  & $Y_{2,h}=\left(-1-bt+(h-1)t^2\right)\,\dfrac{\partial }{\partial t}$  &
 $\begin{array}{l}
 X_1=\dfrac{x-y}{y}\\
X_2=-\dfrac{(x-y)(by+1)}{x}
\end{array}$\\
  \hline
\end{tabular}
\end{center}
}

\medskip

\noindent 3. Objects associated with the recurrence
$u_{n+2}=\frac{b\,(u_{n+1}-u_{n})+u_{n+1}^2}{u_{n}}$ with $b\in\C$:

{\small
\begin{center}
\noindent\begin{tabular}{| l | l | l |}
  \hline
  Associated map:    & First integral:  & Parametrization of  $C_{3,h}$: \\
 $F_3(x,y)=\left(y,\dfrac{-bx+by+y^2}{x}\right)$   &  $V_3(x,y)=\dfrac{y+b}{x}$
 &   $P_{3,h}(t)=\left(t,ht-b\right)$\\
   \hline
      M\"obius map: & Lie symmetry  of $M_{3,h}$:  & Lie symmetry of $F_3$:\\
 $M_{3,h}(t)=ht-b$  & $Y_{3,h}=\left(b+(1-h)t\right)\,\dfrac{\partial }{\partial t}$  &
 $\begin{array}{l}
 X_1=x-y\\
X_2=\dfrac{(x-y)(b+y)}{x}
\end{array}$\\
  \hline
\end{tabular}
\end{center}
}

\medskip

\noindent 4. Objects associated with the recurrence $u_{n+2}=\frac{b
u_{n+1}+u_{n+1}^2}{u_{n}+b}$ with $b\in\C\setminus\{0\}$:

{\small
\begin{center}
\noindent\begin{tabular}{| l | l | l |}
  \hline
  Associated map:    & First integral:  & Parametrization of  $C_{4,h}$: \\
 $F_4(x,y)=\left(y,\dfrac{by+y^2}{x+b}\right)$   &  $V_4(x,y)=\dfrac{x+b}{y}$
 &   $P_{4,h}(t)=\left(t,\dfrac{t+b}{h}\right)$\\
   \hline
      M\"obius map: & Lie symmetry  of $M_{4,h}$:  & Lie symmetry of $F_4$:\\
 $M_{4,h}(t)=\dfrac{t+b}{h}$  & $Y_{4,h}=\left(-\dfrac{b}{h}+
 \left(1-\dfrac{1}{h}\right)t\right)\,\dfrac{\partial }{\partial t}$  &
 $\begin{array}{l}
 X_1=x-y\\
X_2=\dfrac{y(x-y)}{x+b}
\end{array}$\\
  \hline
\end{tabular}
\end{center}
}

\medskip

\noindent 5. Objects associated with the recurrence $u_{n+2}=\frac{b
u_{n+1}+u_n u_{n+1}}{u_{n+1}+b}$ with $b\in\C\setminus\{0\}$:

{\small
\begin{center}
\noindent\begin{tabular}{| l | l | l |}
  \hline
  Associated map:    & First integral:  & Parametrization of  $C_{5,h}$: \\
 $F_5(x,y)=\left(y,\dfrac{by+xy}{y+b}\right)$   &  $V_5(x,y)=y(x+b)$
 &   $P_{5,h}(t)=\left(t,\dfrac{h}{t+b}\right)$\\
   \hline
      M\"obius map: & Lie symmetry  of $M_{5,h}$:  & Lie symmetry of $F_5$:\\
 $M_{5,h}(t)=\dfrac{h}{t+b}$  & $Y_{5,h}=\left(-h+bt+t^2\right)\,\dfrac{\partial }{\partial t}$  &
 $\begin{array}{l}
 X_1=(x-y)(b+x)\\
X_2=-y(x-y)
\end{array}$\\
  \hline
\end{tabular}
\end{center}
}

Observe that the recurrence (\ref{ODEnostra}) studied in Section
\ref{Ss_lanostra}, corresponds with this recurrence when $b=1$.

\newpage

\noindent 6. Objects associated with the recurrence $u_{n+2}=\frac{b
u_{n}-b u_{n+1}+u_n u_{n+1}}{u_{n+1}}$ with $b\in\C\setminus\{0\}$:

{\small
\begin{center}
\noindent\begin{tabular}{| l | l | l |}
  \hline
  Associated map:    & First integral:  & Parametrization of  $C_{6,h}$: \\
 $F_6(x,y)=\left(y,\dfrac{bx-by+xy}{y}\right)$   &  $V_6(x,y)=x(y+b)$
 &   $P_{6,h}(t)=\left(t,\dfrac{-bt+h}{t}\right)$\\
   \hline
      M\"obius map: & Lie symmetry  of $M_{6,h}$:  & Lie symmetry of $F_6$:\\
 $M_{6,h}(t)=\dfrac{-bt+h}{t}$  & $Y_{6,h}=\left(-h+bt+t^2\right)\,\dfrac{\partial }{\partial t}$  &
 $\begin{array}{l}
 X_1=x(x-y)\\
X_2=-(x-y)(y+b)
\end{array}$\\
  \hline
\end{tabular}
\end{center}
}

\subsubsection{Conjugations in the set of Palladino's
maps}\label{Ss_Pal_conju}

In this Section we apply Proposition \ref{P_conjgacions} to detect
the conjugations between the set of Palladino's maps, obtaining:

\begin{propo}\label{P_conju-palladino}
\begin{enumerate}[(a)]
  \item The maps $F_1$, $F_2$, $F_5$ and $F_6$ are birationally
  conjugate.
  \item The maps $F_3$ and $F_4$ are birationally conjugate.
  \item Any map in the set $\{F_1,F_2,F_5,F_6\}$ is not
  conjugate with any map in the set $\{F_3,F_4\}$ via a conjugation
  which is a correspondence between
the respective invariant fibrations $\{C_{j,h}\}$.
\end{enumerate}
\end{propo}

\begin{proof} (a) A computation shows that the maps $M_{1,h}$ and
$M_{2,k}$ defined in the above section are conjugate via
$$m_{h}(t)=\dfrac{h}{b(b^2t+b-h)}$$ with the correspondence between the level sets given
by $$k=f(h):=-\dfrac{b^3-h}{h}.$$ Hence, by using Equation
(\ref{E_conjugacions}) in Proposition \ref{P_conjgacions}, we obtain
that $F_1=\Psi^{-1}F_2\Psi$, where
$$
\Psi(x,y)=\Psi(x,y)=P_{2,f(h)}\circ m_{h}\circ
P_{1,h}^{-1}(x,y)_{\left|
                   h=V_1(x,y)\right.}=\left(-\dfrac{by+1}{b},-\dfrac{bx+1}{b(bx-by+1)}\right),
$$
and
$$\Psi^{-1}(x,y)=P_{1,f^{-1}(k)}\circ m_{f^{-1}(k)}\circ
P_{2,k}^{-1}(x,y)_{\left|
                   k=V_2(x,y)\right.}=\left({-{\frac {{b}^{2}xy+2\,by+1}{b \left( by+1
\right) }} ,-{\frac {bx+1}{b}}}\right).
$$

Analogously, the maps $M_{2,h}$ and $M_{5,k}$ are conjugate via the
map $m_{h}(t)=-{1}/{(t+b)}$, with the correspondence given by
$k=f(h):=h-1$. Hence  $F_2=\Psi^{-1}F_5\Psi$, where $
\Psi(x,y)=\left(-(bx+1)/{x},-(by+1)/{y}\right), $ and
$\Psi^{-1}(x,y)=\left(-1/(x+b),-1/(y+b)\right)$.

The maps $M_{6,h}$ and $M_{5,k}$ are conjugate via the map
$m_{h}(t)=-h/t$ with the correspondence $k=f(h)=h$. Again  we get
$F_6=\Psi^{-1}F_5\Psi$ with $
\Psi(x,y)=\left(-y-b,-{x(y+b)}/{y}\right), $ and
$\Psi^{-1}(x,y)=\left(-y \left( x+b \right)/x,-x-b\right)$.

(b)  The maps $M_{3,h}$ and $M_{4,k}$ are conjugate via the map
$m_{h}(t)={b\,t}/{((h-1)t-b)}$ with the correspondence $k=f(h)=h$,
so $F_3=\Psi^{-1}F_4\Psi$ with $$\Psi(x,y)=P_{4,f(h)}\circ
m_{h}\circ P_{3,h}^{-1}(x,y)_{\left|
                   h=V_3(x,y)\right.}=\left(\frac{b\, x}{y-x},\frac{-b\,xy}{(x-y)(y+b)}\right),
$$ and $\Psi^{-1}(x,y)=\left(bxy/ ( (x-y) (x+b)
),by/(x-y)\right)$.

(c) The statement follows from Proposition \ref{P_conjgacions}, by
taking into account the fact that when we look for  a conjugation
between any of the M\"obius maps $M_i$ with $i=1,2,5,6$, and any
$M_j$ with $j=3,4$, we obtain that there exists such conjugations,
but there is not a bijection between the level sets of $V_3$ and
$V_4$, so it is not possible to construct conjugations between the
maps in the set $\{F_1,F_2,F_5,F_6\}$ and the maps in the set
$\{F_3,F_4\}$, via a conjugation which is a correspondence between
the respective associated invariant fibrations.~\end{proof}

\subsubsection{Analysis of the Palladino's
recurrences number $3$ and $5$}\label{Ss_Pal-analisis}

From Proposition \ref{P_conju-palladino}, the Palladino recurrences
number  $1,2,5$ and $6$ on one hand, and number  $3$ and $4$ on the
other, have the same dynamics from a qualitative viewpoint.  We
characterize the dynamics of a representative of each set of
recurrences, by studying the maps $M_3$ and $M_5$. First, we
consider the difference equation
\begin{equation}\label{E_Pal-3}
u_{n+2}=\dfrac{b\,(u_{n+1}-u_{n})+u_{n+1}^2}{u_{n}}, \mbox{ with }
b\in\C.
\end{equation}
For each initial condition $u_0$, $u_1$ in its good set
$\G\subset\C$,  set $I=V_3(u_0,u_1)=(u_1+b)/u_0$. Then:
\begin{propo}\label{P_P3}
Let $\{u_n\}$ be a solution of Equation (\ref{E_Pal-3}) with initial
condition $u_0,u_1$ in $\G$, then:
\begin{enumerate}
  \item[(a)] If   $|I|<1$, then the solution converges to
  $u=b/(I-1)$.
  \item[(b)] If $|I|>1$ or $I=1$ and $b\neq 0$, then the solution is
  unbounded.
  \item[(c)]  If $I$ is a $p$-root of the unity with $I\neq 1$, then the solution is
  $p$-periodic.
  \item[(d)] If $|I|=1$ and $I$ is not a $p$-root
of the unity, then $\{u_n\}$ is conjugate to a sequence generated by
an irrational rotation of angle $\arg(I)$, and  the set of
accumulation points of $\{u_n\}$ is a the circle in  $\C$ with
center $z=b/(I-1)$ and radius $|u_0-b/(I-1)|$.
  \item[(e)] If $I=1$ and $b=0$, then the solution is constant .
\end{enumerate}
\end{propo}

\begin{proof} The proof follows straightforwardly
 by noticing that due to the particular form of the family of parameterizations
 $\{P_{3,h}\}$,
we have that $u_n=M_{3,h}^n(u_0)$, for $h=V_3(u_0,u_1)$. Now, if
$h=1$, then $M_{3,h}$ is the identity if $b=0$ and each orbit of
$M_{3,1}$ is unbounded if $b\neq 0$.  If $h\neq 1$, then $M_{3,h}$
has a unique fixed point $t_h:=b/(h-1)$. This point is a global
attractor if $|h|<1$, and a repeller (and the orbits are unbounded)
if $|h|>1$. If $|h|=1$ with $h\neq 1$ then $M_{3,h}$ is conjugate to
the rotation given by $z\to h\,z$, with $z\in\C$.~\end{proof}

\smallskip

Now, we consider the difference equation
\begin{equation}\label{E_Pal-5}
u_{n+2}=\frac{b u_{n+1}+u_n u_{n+1}}{u_{n+1}+b},\mbox{ with }
b\in\C\setminus\{0\}
\end{equation}
 For
each initial condition $u_0$, $u_1$ in its good set $\G\subset\C$,
set
$$t_j=t_j(u_0,u_1):=\dfrac{-b+(-1)^j\sqrt{b^2+4\,u_1(u_0+b)}}{2}
\mbox{ and } \xi=-\frac{\left( b+\sqrt {{b}^{2}+4\,u_1(u_0+b)}
\right) ^{2}}{4u_1(u_0+b)}.
$$

\begin{propo}\label{P_P5} Let $\{u_n\}$ be a solution of Equation (\ref{E_Pal-5}) with initial condition $u_0,u_1$ in
$\G$, then:
\begin{enumerate}
  \item[(a)] When $u_1(u_0+b)=-b^2/4$,  the solution converges to
  $u=-b/2$.
  \item[(b)] When $u_1(u_0+b)\neq -b^2/4$, then:
\begin{enumerate}
  \item[(i)] If $u_1=u_0=t_j$, for any $j=1,2$, then the solution is
  constant.
  \item[(ii)] If $|\xi|<1$, $u_0\neq t_1$ and $u_1\neq t_1$, then
  the solution  converges to $u=t_0$; and if $|\xi|>1$, $u_0\neq t_0$ and $u_1\neq t_0$, then
  the solution  converges to $u=t_1$.
  \item[(iii)] If $|\xi|=1$, then the solution is either $p$-periodic  if
  $\xi$ is a $p$-root of unity, or such that
is conjugate to a sequence generated by an irrational rotation of
angle $\arg(\xi)$, and  the set of accumulation points of $\{u_n\}$
is a set homeomorphic to $\mathbb{S}^1$ in  $\widehat{\C}$.
\end{enumerate}
\end{enumerate}
\end{propo}

\begin{proof} The proof is a direct application of Proposition
\ref{P_dinamica-moeb} since, from the particular form of the
parameterizations $P_{5,h}$, we have that $u_n=M_{5,h}^n(u_0)$.
\end{proof}

\newpage

\section*{Appendix: Proper parameterizations and its inversion}\label{A_inversion}

Suppose that $C$ is a rational curve. Observe that $\K(C)[t]$, the
polynomials with coefficients in the field of rational functions in
$C$, is a Euclidean domain, so the Euclidean algorithm can be
applied to compute the greatest common divisor. Let $P(t)$ be a
rational affine parametrization of $C$ over $\K$ defined as
$$P(t)=\Bigg(\frac{P_{11}(t)}{P_{12}(t)},\frac{P_{21}(t)}{P_{22}(t)}\Bigg)\,,$$
where $P_{ij}(t)\in\K[t]$ and $\gcd(P_{1i},P_{2i})=1$, that is
$P(t)$ is in \emph{reduced form}. The next results give a quick way
to check whether a parametrization is proper or not:

\begin{teo}\label{T_proper-or-not}(\cite[Thm. 4.21]{SWPD})
 Let $C$ be an affine rational curve defined over $\K$ with defining polynomial
 $f(x,y)\in\K[x,y]$, and let $P(t)$ be a parametrization
 of $C$. Then $P(t)$ is proper if and only if
 $$deg(P(t))=max\{deg_x(f),deg_y(f)\}\,,$$
 where the degree of $P(t)$ is the maximum of the degrees of its rational
 components.
\end{teo}

The next result allow us to compute the inverse of $P(t)$.

\begin{teo}\label{T_inversa-param}(\cite[Thm. 4.37]{SWPD})
 Let $P(t)$ be a proper parametrization in reduced form with nonconstant components
 of a rational curve $C$. Let $$\begin{array}{l}
                                  {H_1}(t,x):=x\, P_{12}(t)-P_{11}(t),\\
                                  {H_2}(t,y):=y\, P_{22}(t)-P_{21}(t),\\
                                \end{array}$$ be considered as
 polynomials in $\K(C)[t]$. Set $M(x,y,t):=\gcd_{\K(C)[t]}({H_1},{H_2})$, then,
 $\deg_t(M(x,y,t))=1.$
 Moreover, its single root in $t$, is the inverse of $P$.
\end{teo}

As a consequence of the above result $M(x,y,t)$ is a linear
polynomial in $t$, so setting $M(x,y,t)=D_1(x,y)\, t-D_0(x,y)$ the
inverse of the parametrization $P$ is given by
\begin{equation}\label{E_inversa-general}
P^{-1}(t)=\frac{D_0(x,y)}{D_1(x,y)}.
\end{equation}


\begin{thebibliography}{12}



\bibitem{BRcon} G.~Bastien, M.~Rogalski.  \textsl{On some algebraic difference equations
$u_{n+2}u_n=\psi(u_{n+1})$ in $\R^+$, related to families of conics
or cubics: generalization of the Lyness' sequences.} J.~Math. Anal.
Appl. 300  (2004), 303--333.

\bibitem{BR2} G.~Bastien, M.~Rogalski.  \textsl{A biquadratic system of two order one difference equations:
periods, chaotic behavior of the associated dynamical system.} Int.
J. Bifurcations and Chaos 22 (2012), 1250266, 10pp.


\bibitem{CGM1} A.~Cima, A.~Gasull, V.~Ma\~{n}osa.
\textsl{Dynamics of rational discrete dynamical systems via first
integrals.} Int. J. Bifurcations and Chaos 16 (2006), 631-–645.

\bibitem{CGM08} A.~Cima, A.~Gasull, V.~Ma\~{n}osa.
\textsl{Studying discrete dynamical systems through differential
equations.} J. Differential Equations 244 (2008), 630--648.

\bibitem{CM} A.~Cima, F.~Ma\~{n}osas. \textsl{Real dynamics of integrable birational maps.}
Qual. Theory Dyn. Syst. 10 (2011), 247--275.

\bibitem{Diller} J.~Diller. \textsl{Dynamics of birational maps of
$P^2$.} Indiana Univ. Math. J. 45 (1996),  721--772.

\bibitem{DF} J.~Diller, C.~Favre. \textsl{Dynamics of Bimeromorphic Maps of Surfaces}
Amer. J.  Math. 123 (2001), 1135--1169

\bibitem{GM} I.~G\'alvez, V.~Ma\~{n}osa. \textsl{Periodic orbits of planar integrable
birational maps.} Nonlinear Maps and their Applications,
R.~L\'opez-Ru\'\i z et al. eds. Springer Proceedings in Mathematics
\& Statistics, Vol. 112. Springer, Berlin 2015.


\bibitem{HBQC} F.~Haggar, G.B.~Byrnes,  G.R.W.~Quispel, H.W.~Capel.
\textsl{$k$--integrals and $k$--Lie symmetries in discrete dynamical
systems.} Physica A 233 (1996), 379--394.

\bibitem{H} M.~van Hoeij. \textsl{Rational  parameterizations of algebraic curves using a canonical
divisor.} J. Symb. Comput. 23 (1997), 209--227.



\bibitem{JRV}  D.~Jogia, J.~A.~G.~Roberts, F.~Vivaldi.
\textsl{An algebraic geometric approach to integrable maps of the
plane.} J. Physics A: Math. \& Gen. 39 (2006), 1133--1149.




\bibitem{P} F.J.~Palladino. \textsl{On invariants and Forbidden sets.}
arXiv:1203.2170v2.


\bibitem{SS} S.~Saito, N.~Saitoh \textsl{On recurrence equations associated
with invariant varieties of periodic points.} J. Physics A: Math. \&
Gen. 40 (2007), 12775--12787.


\bibitem{SWPD} J.R.~Sendra, F.~Winkler, S.~P\'erez-Diaz. Rational
Algebraic Curves. Springer, New York 2008.

\bibitem{SG} I. Sushko, L. Gardini. \textsl{Center bifurcation of a point on
the Poincar\'e equator.} Grazer Math. Ber. 345 (2009), 254--279.

\end{thebibliography}
\end{document}